\documentclass[12pt,reqno]{amsart}

\usepackage[margin=3.5 cm]{geometry}
\usepackage{graphicx, mathabx}
\usepackage{color}
\newtheorem{theorem}{Theorem}[section]
\newtheorem*{theorem*}{Theorem}   %theorem without numbers
\newtheorem{prop}{Proposition}[section]

\newtheorem{lemma}[theorem]{Lemma}

\theoremstyle{definition}

\theoremstyle{remark}
\newtheorem{remark}[theorem]{Remark}

\usepackage[utf8]{inputenc}
\usepackage[english]{babel}
 
\usepackage{amsthm}
\newcommand{\ud}{\mathrm{d}}

\newtheorem{hyp}{Hypothesis}

\begin{document}

\title[Decay of Eigenfunctions on a bounded below potential]{Agmon-type decay of eigenfunctions for a class of Schr\"{o}dinger operators with non-compact classically allowed region}

\author[C.\ A.\ Marx]{Christoph A.\ Marx}
\address{Department of Mathematics,
Oberlin College,
Oberlin, Ohio 44074, USA}
\email{cmarx@oberlin.edu}

\author[H. \ Zhu ]{Hengrui Zhu}
\address{Oberlin College, Oberlin, Ohio 44074, USA}
\email{hzhu@oberlin.edu}

\maketitle
\begin{abstract}
An important result by Agmon implies that an eigenfunction of a Schr\"{o}dinger operator in $\mathbb{R}^n$ with eigenvalue $E$ below the bottom of the essential spectrum decays exponentially if the associated classically allowed region $\{x \in \mathbb{R}^n~:~ V(x) \leq E \}$ is compact. We extend this result to a class of Schr\"{o}dinger operators with eigenvalues, for which the classically allowed region is not necessarily compactly supported: We show that integrability of the characteristic function of the classically allowed region with respect to an increasing weight function of bounded logarithmic derivative leads to $L^2$-decay of the eigenfunction with respect to the same weight. Here, the decay is measured in the Agmon metric, which takes into account anisotropies of the potential. In particular, for a power law (or, respectively, exponential) weight, our main result implies that power law (or, respectively, exponential) decay of ``the size of the classically allowed region'' allows to conclude power law (or, respectively, exponential) decay, in the Agmon metric, of the eigenfunction.
\end{abstract}

\vspace{0.2 cm}
\section{Introduction}

In a series of lectures given in 1980 at the University of Virginia, later published as \cite{Agmon_lectures1982}, S. Agmon proposed his now celebrated method for proving exponential decay of eigenfunctions for Schr\"odinger type operators $H=-\Delta + V$ in $\mathbb{R}^n$ for eigenvalues below the bottom of the essential spectrum. One of the main accomplishments of Agmon's approach was to account for {\em{anisotropies}} of the potential function $V$ captured through a pseudo-metric, now known as the {\em{Agmon metric}}; see (\ref{eq:agmondef}) below. While the spectral condition on the eigenvalue $E$ can be expressed more generally as {\em{Agmon's $\lambda$-condition}} (see (1.7) in \cite{Agmon_lectures1982}), it is typically formulated by assuming compactness of the classically allowed region $\{x \in \mathbb{R}^n ~:~ V(x) \leq E + \delta \}$, for some $\delta > 0$. In this situation, Persson's characterization of the essential spectrum \cite{Persson_1960_essentialsp} (see also \cite[Theorem 14.11]{Hislop_Sigal_book_1996}) implies that $E$ is strictly below the bottom of the essential spectrum (separated by at least $\delta$). We refer the reader to \cite[chapter 3]{Hislop_Sigal_book_1996} or \cite[chapter 8.5]{Pankov_SpectralTheory_book_2001} for a pedagogical treatment of Agmon's method for this basic set-up.

In the 40 years since, Agmon's result has triggered a vast body of literature that develops further the subject of proving decay-estimates. An extensive discussion of both the history and a summary of the works up to the year 2000 can be found in the review article \cite{Hislop_Review_2000}. Of the developments in more recent years, we mention without aiming for a comprehensive list: applications to tight binding type models (see e.g. \cite{Fefferman_Weinstein_CMP_2020, Wang_Zhang_AHP_2020}), decay estimates for magnetic Schr\"odinger operators (see e.g. \cite{Mayboroda_Poggi_2019, Raymond_book_2017}), decay estimates in superconductivity (for a review, see \cite{Fournais_Helffer_book_2010}) and for the Robin problem (see e.g. \cite{Helffer_Kachmar_Raymond_CommContMath_2015, Hellfer_Pankrashkin_JLondMathSoc_2015}), Agmon estimates on quantum graphs (see e.g. \cite{Harrel_Maltsev_2020_TransAMS, Akduman_Pankov_2018, Harrel_Maltsev_2018_CMP}), and the landscape function approach (see e.g. \cite{ADFJM_2019, Steinerberger_ProcAMS_2017, Filoche_Mayboroda_ProcNatlAcadSci_2012}). 

The subject of this short paper is to examine anisotropic decay for the basic Schr\"odinger set-up in $\mathbb{R}^n$, without assuming compactness of the classically allowed region. Specifically, we show that if the characteristic function of the classically allowed region associated with an eigenvalue $E$ is merely integrable with respect to (the square of) an increasing weight function $1 \leq \phi\in C^1([0,+\infty))$ {\em{with bounded logarithmic derivative}}, i.e. if for some $0 < \epsilon < 1$ and $\delta > 0$ one has 
\begin{equation} \label{eq:integrability}
    \int \chi_{\{V\leq E+\delta\}}(x) ~\phi\left((1-\epsilon)\rho_E(x)\right)^2 ~\ud^n x < +\infty ~\mbox{,}
\end{equation}
then the associated eigenfunctions exhibit $L^2$-decay with respect to the same weight. Here, the decay is measured in the Agmon distance $\rho_E(x)$ of $x$ to the origin (see (\ref{eq:agmondef})-(\ref{eq:agmondef_origin}) for the definition), which takes into account anisotropies of the potential. Our main result is formulated in Theorem \ref{the:1}, the precise set-up of which is summarized in hypotheses (H\ref{hyp:1}) and (H\ref{hyp:2}) of section \ref{sec:set-up}. Particularly relevant examples for admissible weight functions $\phi$ in (\ref{eq:integrability}) are power functions $\phi(t) = (1 + t)^r$ with $r >0$, and exponentials $\phi(t) = \mathrm{e}^{\alpha t}$ with $\alpha > 0$. In this context we also point out that our integrability condition (\ref{eq:integrability}) still implies that the eigenvalue $E$ is at a positive distance of at least $\delta$ below the bottom of the essential spectrum, see Proposition \ref{prop:essentialspectrum} in section \ref{sec:L2:sub_conclremarks_ess}.

Morally, the integrability condition in (\ref{eq:integrability}) can be understood as allowing for a classically allowed region which stretches out to $\infty$, however in such a way that its measure decays like $1/\phi$, when measured in the Agmon metric; see also Remark \ref{rem:integrabilitycond_1d} for the situation in one dimension. In particular, for power law weights in (\ref{eq:integrability}), $\phi(t) = (1 + t)^r$ with exponent $r > 0$, our main result in Theorem \ref{the:1} implies that power law decay of ``the size of the classically allowed region'' still allows to conclude power law decay with same exponent of the wave function. Explicit examples for potentials satisfying the condition (\ref{eq:integrability}) are constructed in section \ref{sec:L2:sub_conclremarks_example}. We note that, since for compact classically allowed region, exponential decay of eigenfunctions is known to be optimal \cite{Carmona_Simon_CMP_1981_lowerbounds}, our condition of bounded logarithmic derivative for the weight function $\phi$ in (\ref{eq:integrability}) is in general necessary. Finally, we mention that while our result could be related the more general framework of Agmon's $\lambda$-condition, we believe that a direct proof is valuable in its own right; to our knowledge, this type of question has not been examined explicitly in the literature. 

We structure the paper as follows: In section \ref{sec:set-up}, we describe the precise set-up (see hypotheses (H\ref{hyp:1}) - (H\ref{hyp:2}), and (H\ref{hyp:3}), respectively) for our main results, Theorem \ref{the:1} and its extension, Theorem \ref{the:2}. The proofs of Theorems \ref{the:1} and \ref{the:2} are given in section \ref{sec:L2:sub_proofs}. Section \ref{sec:L2:sub_conclremarks} provides some additional context for our main results: In section \ref{sec:L2:sub_conclremarks_ess}, we show that for eigenvalues $E$ for which the integrability condition in (\ref{eq:integrability}) is satisfied are still strictly below the bottom of the essential spectrum (see Proposition \ref{prop:essentialspectrum}). In section \ref{sec:L2:sub_conclremarks_example}, we construct explicit examples of potentials for which the integrability condition in (\ref{eq:integrability}) holds. Finally, in section \ref{sec:ptw}, we prove point-wise decay of eigenfunctions for more regular potentials (Theorem \ref{thm:4}); this extends the well known result for an exponential weight going back to Agmon \cite{Agmon_lectures1982} to the more general weight functions with bounded logarithmic derivative, considered in this article.

\vspace{.2in}
\noindent
\textbf{Acknowledgements:} The authors would like to thank Peter D. Hislop for numerous valuable discussions while preparing this manuscript.

\section{Set-up and main result} \label{sec:set-up}

Let $H = -\Delta+V$ be a closed operator acting on $L^2(\mathbb{R}^n)$ with $\sigma(H)\subseteq \mathbb{R}$, where $V$ is real-valued, continuous, and bounded below. While certain aspects of this paper may hold true for more general situations, for simplicity and concreteness, we consider the following set-up:
\begin{hyp}[H\ref{hyp:1}] \label{hyp:1}
\begin{itemize}
\item[(a)] Let $V$ be real-valued continuous and bounded below,
\begin{equation} \label{eq:Vlowerbound}
-\infty < m_V := \inf_{x\in\mathbb{R}^n}{V(x)} ~\mbox{,}
\end{equation}
and suppose that, as a multiplication operator with natural self-adjoint domain $D(V):=\{f\in L^2(\mathbb{R}^n): (V f)\in L^2\}$, $V$ is relatively $(-\Delta)$-bounded with relative bound strictly less than one; here, as usual, we define $(-\Delta)$ as a self-adjoint operator on $H^2(\mathbb{R}^n$). In this case, $H=-\Delta + V$ is a self-adjoint operator on the domain $H^2(\mathbb{R}^n) \subseteq D(V)$. 
\item[(b)] Let $\psi$ be an eigenfunction of $H$ with associated eigenvalue $E\in\mathbb{R}$ which also satisfies that $\psi \in L^{\infty}(\mathbb{R}^n)$. 
\end{itemize}
\end{hyp}
\begin{remark} \label{rem:linftyassumption}
In view of hypothesis (H\ref{hyp:1}) part (b), we note that for $n \leq 3$, the Sobolev Embedding Theorem embeds $H^2(\mathbb{R}^n)$ into the continuous functions vanishing at infinity, so that the condition $\psi\in L^\infty(\mathbb{R}^n)$ holds trivially for any eigenfunction $\psi$ of $H$. Moreover, in any dimension $n \in \mathbb{N}$, an eigenfunction $\psi$ of $H$ also satisfies $\psi\in L^\infty(\mathbb{R}^n)$ if the potential $V$ is sufficiently regular (with regularity depending on $n$); indeed, $V \in \mathcal{C}^{k}(\mathbb{R}^n)$ with bounded derivatives, implies that any eigenvector $\psi$ of $H$ automatically satisfies $\psi \in H^{k + 2}(\mathbb{R}^n)$ and thus is continuous and vanishing at infinity if $k + 2 > \frac{n}{2}$, see e.g. \cite[Proposition 1.2]{Hislop_Review_2000}. The latter will also play a role in the point-wise bounds discussed in section \ref{sec:ptw}.
\end{remark}

We recall that for a continuous potential $V$ and $E \in \mathbb{R}$, the {\em{Agmon metric}} is defined by
\begin{equation} \label{eq:agmondef}
    \rho_E(x,y) := \inf_{\gamma \in P_{x,y}}\int_0^1\left(V(\gamma(t))-E\right)_+^{1/2} \vert \dot{\gamma}(t)\vert ~\mathrm{d} t~,
\end{equation}
where $P_{x,y}:=\{\gamma:[0,1]\to \mathbb{R}^n|\gamma(0)=x,~\gamma(1)=y,~\mbox{and}~ \gamma \in AC[0,1]\}$, and $(V-E)_+ := \max\{V-E,0\}$. In particular, the distance in the Agmon metric is zero precisely if two points can be connected by a path which remains completely inside the classically allowed region $\{V \leq E\}$. We note that in one dimension, the definition in (\ref{eq:agmondef}) reduces to the WKB factor: 
\begin{equation} \label{eq:agmon_1d}
    \rho_E(x,y) = \mathrm{sgn}(y-x) \int_x^y\left(V(t)-E\right)_+^{1/2} ~\mathrm{d} t ~\mbox{, for $x, y \in \mathbb{R}^1$ .}
\end{equation}
We will write
\begin{equation} \label{eq:agmondef_origin}
\rho_E(x):= \rho_E(x,0) ~\mbox{,}
\end{equation}
for the Agmon distance of $x \in \mathbb{R}^n$ to the origin.

Finally, we recall the key property of the Agmon metric which provides a relation to the classically allowed region: $x \mapsto \rho_E(x)$ is locally Lipschitz on $\mathbb{R}^n$, in particular, for a.e. in $x \in \mathbb{R}^n$, $\rho_E(x)$ is differentiable with
\begin{equation}\label{eq:agmon}
    \left\vert\nabla \rho_E(x)\right\vert^2\leq (V(x)-E)_+~,
\end{equation}
see, e.g., Proposition 3.3 in \cite{Hislop_Sigal_book_1996}. 

\begin{hyp}[H\ref{hyp:2}] \label{hyp:2}
Let $E \in \mathbb{R}$ be an eigenvalue satisfying hypothesis (H\ref{hyp:1}). For some function $1 \leq \phi\in C^1([0,+\infty))$ with properties
\begin{equation}\label{eq:phi_1}
    \lim_{t \to +\infty}{\phi(t)}=\infty ~\mbox{ and} 
\end{equation}
\begin{equation}\label{eq:mphi}
    0<\sup \left\vert\frac{\phi'}{\phi}\right\vert=:M_\phi<\infty~,
\end{equation}
suppose there exist $\epsilon , \delta > 0$ with 
\begin{equation}\label{eq:epsilon}
    \max\left\{0,1-M_\phi^{-2}\right\}<\epsilon<1~~ 
\end{equation}
such that
\begin{equation}\label{eq:l_2}
    \chi_{\{V\leq E+\delta\}}\phi\left((1-\epsilon)\rho_E\right)\in L^2(\mathbb{R}^n)~.
\end{equation}
Here, $\rho_E(x)$ is the distance to the origin in the Agmon metric, associated with the eigenvalue $E$, defined in (\ref{eq:agmon_1d})-(\ref{eq:agmondef_origin}).
\end{hyp}

We will refer to a function $1 \leq \phi\in C^1([0,+\infty))$ which satisfies the conditions (\ref{eq:phi_1})-(\ref{eq:mphi}) as an {\em{admissible weight function}}. Note that condition (\ref{eq:mphi}) is equivalent to assuming that $\phi$ has {\em{bounded logarithmic derivative}}. In view of our main result, particularly relevant examples for admissible weight functions are power functions $\phi(t) = (1 + t)^r$ with $r >0$, and exponentials $\phi(t) = \mathrm{e}^{\alpha t}$ with $\alpha > 0$. Since for compact classically allowed region, exponential decay of eigenfunctions is known to be optimal \cite{Carmona_Simon_CMP_1981_lowerbounds}, our condition of bounded logarithmic derivative for the weight function $\phi$ in (\ref{eq:mphi}) is in general necessary. Condition (\ref{eq:l_2}) in hypothesis (H\ref{hyp:2}) thus requires that the ``$\delta$-enlarged'' classically allowed region $\{ V \leq E + \delta \}$ be $\phi$-integrable in the Agmon metric.

\begin{remark} \label{rem:integrabilitycond_1d}
To further interpret the meaning of the integrability condition in (\ref{eq:l_2}), observe that in one dimension ($n=1$), continuity of $V$ implies that the sub-level set $\{V \leq E+\delta\}$ is a countable union of closed intervals, i.e.
\begin{equation} \label{eq:interpretintegrabilitycond_1d_classicalregion}
\{V \leq E+\delta\} = \bigcup_{j \in \mathbb{Z}} [a_j , b_j] ~\mbox{, }
\end{equation}
with $[a_j, b_j]$ mutually disjoint except possibly for endpoints. Here, without loss of generality, we take $[a_j , b_j] \subseteq [0, + \infty)$ for $j \geq 0$ and $[a_j , b_j] \subseteq (- \infty, 0]$ for $j <0$.

Observe that the simple form of the Agmon metric in one dimension in (\ref{eq:agmon_1d}) implies that $\rho_E(x)$ increases on $[0,+ \infty)$ and decreases on $(-\infty, 0]$ (not necessarily strictly). Thus, since $\phi$ is increasing, the integrability condition (\ref{eq:l_2}) requires that the length of each of the intervals in the decomposition (\ref{eq:interpretintegrabilitycond_1d_classicalregion}) be summable with respect to $\phi((1-\epsilon) \rho_E)^2$, i.e. 
\begin{align} \label{eq:interpretintegrabilitycond_1d_summability}
\sum_{j \geq 0} \phi((1- \epsilon) \rho_E(a_j))^2 & ~\vert b_j - a_j \vert \leq \Vert \chi_{\{V\leq E+\delta\} \cap [0, +\infty) } ~\phi\left((1-\epsilon)\rho_E\right) \Vert_2^2 \nonumber \\
& \leq \sum_{j \geq 0} \phi((1- \epsilon) \rho_E(b_j))^2 ~\vert b_j - a_j \vert ~\mbox{,}
\end{align} 
and similarly for $\Vert \chi_{\{V\leq E+\delta\} \cap (-\infty,0]) } ~\phi\left((1-\epsilon)\rho_E\right) \Vert_2^2$, where $j < 0$ and the roles of $a_j$ and $b_j$ in (\ref{eq:interpretintegrabilitycond_1d_summability}) are interchanged. 
\end{remark}

Our result shows that $\phi$-integrability of the classically allowed region in the sense of (\ref{eq:l_2}) in hypothesis (H\ref{hyp:2}) implies $L^2$-decay of the eigenfunction with respect to the same weight, specifically:
\begin{theorem} \label{the:1}
Let $\psi \in L^2(\mathbb{R}^n)$ be an eigenfunction of $H$ satisfying hypothesis (H\ref{hyp:1}). Suppose that for an associated eigenvalue $E$ and a weight function $\phi$ as in hypothesis (H\ref{hyp:2}), the integrability condition (\ref{eq:l_2}) holds for some $0< \epsilon < 1$ as in (\ref{eq:epsilon}) and $\delta > 0$. Then, there exists a constant $0<c_{\epsilon, \delta} <\infty$, such that
\begin{equation*}
    \int \phi((1-\epsilon)\rho_E(x))^2|\psi(x)|^2 ~\ud^n x\leq c_{\epsilon, \delta} ~.
\end{equation*}
\end{theorem}

We note that the positive lower bound on the value of $\epsilon$ in (\ref{eq:epsilon}) is a result of isolating the contribution of the ``$\delta$-enlarged classically allowed region,'' $\{V < E + \delta\}$; see also (\ref{eq:inner_p}) in the proof below. It drops to zero if $M_\phi\leq 1$, i.e. if $\phi' \leq \phi$ for all $t \in [0, + \infty)$. By Gr\"onwall's inequality, the latter is equivalent to $\phi(x) \leq \mathrm{e}^t $, for all $t\in[0,\infty)$. For the power law weights 
\begin{equation} \label{eq:polyweights_cond}
\phi(t) = (1 + t)^r ~\mbox{ with $r > 1$ ,}
\end{equation}
the condition in (\ref{eq:epsilon}) however implies the positive lower bound
\begin{equation}
\epsilon > 1 - \frac{1}{r^2} ~\mbox{.}
\end{equation}

This positive lower bound on $\epsilon$ can however be removed by replacing hypothesis (H\ref{hyp:2}) with the following:
\begin{hyp}[H\ref{hyp:3}]\label{hyp:3}
\begin{itemize}
\item[(a)] Let $1 \leq \phi\in C^1([0,+\infty))$ satisfy (\ref{eq:phi_1}) and (\ref{eq:mphi}), as well as that
\begin{equation}\label{eq:p'p}
    \lim_{x\to\infty}\frac{\phi'(x)}{\phi(x)} = 0~.
\end{equation}
\item[(b)] For $V$ and $E \in \mathbb{R}$ as in (H\ref{hyp:1}), assume that
\begin{equation}\label{eq:rho2inf}
    \lim_{x\to\infty}\rho_E(x) = \infty~.
\end{equation}
Suppose that there exist $0<\epsilon<1$ and $\delta>0$ such that (\ref{eq:l_2}) holds.
\end{itemize}
\end{hyp}
An important example of weights satisfying (\ref{eq:p'p}) are the power law weights in (\ref{eq:polyweights_cond}). We mention that in one dimension, the condition (\ref{eq:rho2inf}) is always satisfied; see also Remark \ref{remark:agmon1d_infty}. Replacing hypothesis (H\ref{hyp:2}) by (H\ref{hyp:3}), we then obtain an improved result which removes the positive lower threshold on $\epsilon$ in (\ref{eq:epsilon}):
\begin{theorem}\label{the:2}
Suppose $\psi \in L^2(\mathbb{R}^n)$ is an eigenfunction of $H$ in the sense of hypothesis (H\ref{hyp:1}) and that for an associated eigenvalue $E$ the condition (\ref{eq:rho2inf}) in hypothesis (H\ref{hyp:3}) holds. Given a weight function $\phi$, satisfying the hypothesis (H\ref{hyp:3}) item (a), one has: for each $0<\epsilon<1$ and $\delta > 0$ for which the integrability condition (\ref{eq:l_2}) holds, there exists a constant $0<\tilde{c}_{\epsilon, \delta} <\infty$ such that
\begin{equation*}
    \int \phi((1-\epsilon)\rho_E(x))^2|\psi(x)|^2 ~\ud^n x\leq \tilde{c}_{\epsilon,\delta} ~.
\end{equation*}
\end{theorem}

\section{$L^2$ decay} \label{sec:L^2}

\subsection{Proof of the main results} \label{sec:L2:sub_proofs}

Following Agmon, the key to proving Theorem \ref{the:1} is to perform a gauge transform, which we will adapt to the weight $\phi$. Specifically, given $\epsilon > 0$ as in (\ref{eq:epsilon}), for $\alpha > 0$, we consider the operator
 \begin{equation} \label{eq:gauge_op}
  H_{f_\alpha}:= \phi(f_\alpha)H\phi(f_\alpha)^{-1} \mbox{, } 
\end{equation}
where
\begin{equation}
f_\alpha:=\frac{(1-\epsilon)\rho_E}{1+\alpha (1-\epsilon)\rho_E} ~\mbox{;}
\end{equation}
To avoid confusion, we mention that $\phi(f_\alpha)^{-1}$ denotes $\frac{1}{\phi(f_\alpha)}$ and {\em{not}} the inverse function of $\phi(f_\alpha)$. Since we will eventually take the limit $\alpha \to 0^+$, we also note that
\begin{equation}
f_\alpha(x) \nearrow (1 - \epsilon) \rho_E(x) =: f_0(x) ~\mbox{, }
\end{equation}
{\em{monotonically}} for each $x \in \mathbb{R}^n$.

Finally, given an eigenfunction $\psi$ as in hypothesis (H\ref{hyp:1}), we associate with the transformed operator (\ref{eq:gauge_op}) the transformed wave function:
\begin{equation} \label{eq:gauge_wavefcn}
\Phi_\alpha:=\phi(f_\alpha)\psi ~\mbox{.}
\end{equation}

We then prove the following key lemma:
\begin{lemma}\label{lm:1}
Given $\max\left\{0,1-M_\phi^{-2}\right\}<\epsilon<1$ and $\delta>0$ as in hypothesis (H\ref{hyp:2}). For each $\alpha>0$, one then has that $\Phi_\alpha\in H^1(\mathbb{R}^n)\cap D(V)$ and 
\begin{equation}\label{eq:lm2}
    ||\Phi_\alpha||^2_2 \leq \frac{1}{\eta_\epsilon\delta}\left\{\Re\langle \Phi_\alpha, (H_{f_\alpha}-E)\Phi_\alpha\rangle+\int|\Phi_\alpha|^2(V-E)_-\right\}+\int_{\{V\leq E+\delta\}}|\Phi_\alpha|^2~,
\end{equation}
where 
\begin{equation}
0 < \eta_\epsilon := 1 - M_\phi^2(1-\epsilon) <1 ~\mbox{.}
\end{equation}
\end{lemma}
Here, as common, we denote $(V-E)_- :=\vert V - E \vert - (V-E)_+$.

\begin{proof}
Fix $\alpha > 0$. First, we verify that $\Phi_\alpha \in H^1(\mathbb{R}^n)$. Observe that $\Phi_\alpha \in L^2$ since $\psi\in L^2$ and $\phi(f_\alpha)$ is bounded by
\begin{equation}\label{eq:f_abdd}
    0\leq f_{\alpha}\leq \min{\left\{f_0,\frac{1}{\alpha}\right\}} ~.
\end{equation}
We also need to show that $\nabla\Phi_\alpha\in L^2$, where
\begin{equation}\label{eq:dphi}
    \nabla\Phi_\alpha =\phi'(f_\alpha) ~\psi ~\nabla f_\alpha + \phi(f_\alpha) ~\nabla\psi  ~\mbox{.}
\end{equation}

Using (\ref{eq:agmon}), for a.e. $x \in \mathbb{R}^n$, $\vert \nabla f_\alpha \vert^2$ is estimated by
\begin{equation}\label{eq:df_a}
    \left\vert\nabla f_\alpha\right\vert^2 = \left\vert\frac{\nabla f_0}{(1+\alpha f_0)^2}\right\vert^2 \leq |\nabla f_0 |^2 \leq (1-\epsilon)(V-E)_+ ~\mbox{. }
\end{equation}
Therefore, the bound in (\ref{eq:f_abdd}) and the hypothesis in (H\ref{hyp:1}) that $\psi \in D(V)$ yield
    \begin{align}
        \int |\phi' &(f_\alpha) \psi |^2 ~|\nabla f_\alpha|^2 \leq (1-\epsilon)\sup_{0 \leq  t \leq 1/\alpha}|\phi'(t)|^2\int|\psi|^2(V-E)_+ \nonumber \\
        & =(1-\epsilon)\sup_{0 \leq t \leq 1/\alpha}|\phi'(t)|^2\left\{\int_{\{ V-E \geq 1 \}}|\psi|^2(V-E)_++\int_{\{0 \leq V-E < 1\}}|\psi|^2(V-E)_+ \right\}  \nonumber \\
        &\leq \left. \sup_{0 \leq t \leq 1/\alpha}|\phi'(t)|^2 ~\left\{||\psi(V-E)||^2_{2}+||\psi||^2_{2}\right\}  \right. <\infty~.
    \end{align}
To see that also the second term in (\ref{eq:dphi}) is in $L^2$, observe that by (\ref{eq:f_abdd}), $\phi(f_\alpha)$ is bounded, and that the hypotheses in (H\ref{hyp:1}) imply that $\psi\in H^2(\mathbb{R}^n) \subseteq D(V)$. In summary, we conclude that $\Phi_\alpha\in H^1(\mathbb{R}^n)$; moreover, $\Phi_\alpha\in D(V)$ because $\psi\in D(V)$ and $\phi(f_\alpha)\in L^\infty$.

Now, we compute the gauge transform:
\begin{equation*}
\begin{aligned}
    H_{f_\alpha}&=-\phi(f_\alpha)\Delta\phi(f_\alpha)^{-1}+ V \\
    &= -\left[\phi(f_\alpha)\nabla\phi(f_\alpha)^{-1}\right]\cdot \left[\phi(f_\alpha)\nabla\phi(f_\alpha)^{-1}\right]+ V~,
\end{aligned}
\end{equation*}
where 
\begin{align*}
    \phi(f_\alpha)\nabla\phi(f_\alpha)^{-1} &= \nabla - \frac{\phi'(f_\alpha)}{\phi(f_\alpha)}\nabla f_\alpha = : \nabla - \omega_\alpha~.
\end{align*}

Whence, we have 
\begin{equation}\label{eq:H_f}
    H_{f_\alpha} = - (\nabla - \omega_\alpha)^2 + V  = - \Delta + [\nabla,\omega_\alpha]_+ - \vert\omega_\alpha\vert^2 +V~\mbox{,}
\end{equation}
where $[\nabla,\omega_\alpha]_+$ is the anti-commutator of $\nabla$ and $\omega_\alpha$. Then, since $\Phi_\alpha\in H^1(\mathbb{R}^n)\cap D(V)$, we estimate, using (\ref{eq:H_f}) and (\ref{eq:df_a}):
\begin{align} \label{eq:inner_p}
    \Re\langle \Phi_\alpha, (H_{f_\alpha}-E)\Phi_\alpha\rangle & \geq \langle \Phi_\alpha, (V-|\omega_\alpha|^2-E)\Phi_\alpha\rangle \nonumber \\
    &\geq \langle \Phi_\alpha, (V- \left\vert\frac{\phi'(f_\alpha)}{\phi(f_\alpha)}\right\vert^2(1-\epsilon)(V-E)_+-E)\Phi_\alpha\rangle \nonumber \\
    &\geq \langle \Phi_\alpha, (V- M_\phi^2(1-\epsilon)(V-E)_+-E)\Phi_\alpha\rangle ~\mbox{.}
\end{align}
Here, we used that $[\nabla,\omega_\alpha]_+$ is antisymmetric, so its real part vanishes. Thus, substituting for $M_\phi^2(1-\epsilon)=1-\eta_\epsilon$, and taking $\delta>0$ as in (H\ref{hyp:2}) so that (\ref{eq:l_2}) holds, we have
\begin{align}
\Re\langle & \Phi_\alpha , (H_{f_\alpha}-E) \Phi_\alpha\rangle  \geq  \langle \Phi_\alpha, (-(V-E)_- +\eta_\epsilon(V-E)_+)\Phi_\alpha\rangle \nonumber \\
& =\eta_\epsilon \left\{\int_{\{V> E+\delta\}}|\Phi_\alpha|^2(V-E)_+ + \int_{\left\{E < V\leq E+\delta\right\}}|\Phi_\alpha|^2(V-E)_+ \right\} \nonumber  \\
 & -\int |\Phi_\alpha|^2(V-E)_- \geq \eta_\epsilon\delta\int_{\{V>E+\delta\}}|\Phi_\alpha|^2-\int |\Phi_\alpha|^2(V-E)_- \nonumber  \\
    & = \eta_\epsilon\delta\left\{||\Phi_\alpha||^2_2-\int_{\{V\leq E+\delta\}}|\Phi_\alpha|^2\right\}-\int |\Phi_\alpha|^2(V-E)_-~. \label{eq:inner}
\end{align}

Finally, to show the finiteness of the last two terms on the right hand side of  (\ref{eq:inner}), we write
\begin{equation}\label{eq:S}
    S_{E,V,\epsilon,\delta}:=||\chi_{\{V\leq E+\delta\}}\phi((1-\epsilon)\rho_E)||_2^2<\infty~,
\end{equation}
whose existence is guaranteed by (\ref{eq:l_2}) in (H\ref{hyp:2}). Then, we estimate, using our hypotheses in (H\ref{hyp:1}) that $V$ is bounded below and that $\psi \in L^\infty$:
    \begin{align}
    \int |\Phi_\alpha|^2(V-E)_- & \leq \int_{\{V\leq E\}}|\Phi_\alpha|^2(E-V) = \int \chi_{\{V\leq E\}}^2|\psi|^2\phi(f_\alpha)^2(E-V) \nonumber \\
     & \leq\int \chi^2_{\{V\leq E+\delta\}}|\psi|^2\phi(f_0)^2(E-m_V) \nonumber \\
    & \leq (E-m_V) ||\psi||^2_\infty S_{E,V,\epsilon,\delta} =: C_1(E,m_V,S_{E,V,\epsilon,\delta}) < \infty~. \label{eq:bd1}
    \end{align}
Similarly, we have
    \begin{align}
    \int_{\{V \leq E+\delta\}} |\Phi_\alpha|^2 & \leq ||\psi||^2_\infty S_{E,V,\epsilon,\delta} =: C_2(E,m_V,S_{E,V,\epsilon,\delta})<\infty~. \label{eq:bd2}
    \end{align}
For later purposes, note that both the bounds in (\ref{eq:bd1}) and (\ref{eq:bd2}) are uniform in $\alpha >0$. In summary, combining (\ref{eq:inner}) with (\ref{eq:bd1}) - (\ref{eq:bd2}), we thus arrive at the claim in (\ref{eq:lm2}).
\end{proof}

Now, we are ready to prove the main result in Theorem \ref{the:1}:
\begin{proof}[Proof of Theorem \ref{the:1}]
First, from Lemma \ref{lm:1} and the bounds in (\ref{eq:bd1}) - (\ref{eq:bd2}), we know that  
    \begin{align} \label{eq:bd3}
        ||\Phi_\alpha||^2_2 &\leq \frac{1}{\eta_\epsilon\delta}\left\{\Re\langle \Phi_\alpha, (H_{f_\alpha}-E)\Phi_\alpha\rangle+\int|\Phi_\alpha|^2(V-E)_-\right\}+\int_{\{V\leq E+\delta\}}|\Phi_\alpha|^2 \nonumber \\
        & \leq\frac{1}{\eta_\epsilon\delta}\Re\langle \Phi_\alpha, (H_{f_\alpha}-E)\Phi_\alpha\rangle +\frac{C_1(E,m_V,S_{E,V,\epsilon,\delta})}{\eta_\epsilon\delta}+C_2(E,m_V,S_{E,V,\epsilon,\delta})~. 
    \end{align}

Observe that
\begin{equation}\label{eq:re_0}
    \langle \Phi_\alpha,(H_{f_\alpha}-E)\Phi_\alpha\rangle = \langle \phi(f_\alpha)^2\psi, (H-E)\psi\rangle = 0~,
\end{equation}
hence, by (\ref{eq:bd3}) and (\ref{eq:re_0}), we have:
\begin{equation}
    ||\Phi_\alpha||_2^2 \leq\frac{C_1(E,m_V,S_{E,V,\epsilon,\delta})}{\eta_\epsilon\delta}+C_2(E,m_V,S_{E,V,\epsilon,\delta}) =:c_{\epsilon, \delta}<\infty~.
\end{equation}
Since $c_{\epsilon,\delta}$ is {\em{independent}} of $\alpha>0$, we can take the limit  $\alpha\to0^+$, and thus arrive at the claim by Fatou's Lemma,
\begin{equation*}
    \int \phi((1-\epsilon)\rho_E(x))^2|\psi(x)|^2 \leq\liminf_{\alpha\to0+}||\Phi_\alpha||^2_2 \leq c_{\epsilon, \delta}<\infty~.
\end{equation*}
\end{proof}

We turn to the proof of Theorem \ref{the:2} which removes the positive lower bound on $\epsilon$ in (\ref{eq:epsilon}) by replacing hypothesis (H\ref{hyp:2}) with (H\ref{hyp:3}). As mentioned earlier in section \ref{sec:set-up}, in dimension one (\ref{eq:rho2inf}) is redundant:
\begin{remark} \label{remark:agmon1d_infty}
Note that (\ref{eq:rho2inf}) is {\em{always satisfied in dimension one}} ($n=1$), in which case the definition of the Agmon metric reduces to the WKB factor in (\ref{eq:agmon_1d}).
Indeed, by the integrability condition (\ref{eq:l_2}), we know that
\begin{equation}
    |\{V\leq E+\delta\}\cap \mathbb{R}^\pm|<||\chi_{\{V\leq E+\delta\}}||_2<||\chi_{\{V\leq E+\delta\}}\phi((1-\epsilon)\rho_E)||_2< \infty~,
\end{equation}
whence
\begin{equation}
    \lim_{x\to\pm\infty}|\{V>E+\delta\}\cap[0,x]| = |\{V>E+\delta\}\cap\mathbb{R}^\pm| = \infty~.
\end{equation}
Hence, using the special form of the Agmon metric in $\mathbb{R}$ given in (\ref{eq:agmon_1d}), we have
\begin{equation}
    \begin{aligned}
    \liminf_{x\to\pm\infty}\rho_E(x) &=\liminf_{x\to\pm\infty}\int_0^x(V-E)_+^{1/2}\ud t \\
    &\geq \liminf_{x\to\pm\infty}\int_{\{V>E+\delta\}\cap[0,x]}(V-E)_+^{1/2}\ud t \\
    &\geq \sqrt{\delta}|\{V>E+\delta\}\cap\mathbb{R}^\pm|=\infty
    \end{aligned}
\end{equation}

Equation (\ref{eq:rho2inf}) may however no longer hold in higher dimensions, $\mathbb{R}^n$ with $n \geq 2$, where the more complex form of the Agmon metric in (\ref{eq:agmondef}) may lead to the situation
\begin{equation}
    \liminf_{x\to\infty}\rho_E(x)<\limsup_{x\to\infty}\rho_E(x)~.
\end{equation}
\end{remark}

The additional hypotheses formulated in (H\ref{hyp:3}) will allow us to remove the positive lower bound on $\epsilon$ in (\ref{eq:epsilon}) by modifying $\Phi_\alpha$ in Lemma \ref{lm:1} with a smooth cut-off $\chi_R$: Specifically, consider
\begin{equation}\label{eq:phi_ar}
    \Phi_{\alpha,R}:=\chi_R\Phi_\alpha~ = \chi_R\phi(f_\alpha)\psi~,
\end{equation}
where $\chi_R\in C^\infty(\mathbb{R}^n)$ satisfies $0\leq\chi_R\leq 1$, $\chi_R|_{B_R(0)}=0$, $\chi_R|_{\mathbb{R}^n \setminus \overline{B_{R+1}(0)}}=1$, and that $\nabla \chi_R$ is compactly supported with $||\nabla \chi_R||_\infty\leq 1$. Here, taking advantage of (\ref{eq:p'p}), we choose $R>0$ in (\ref{eq:phi_ar}) so that
\begin{equation} \label{eq:defnR}
    \sup_{x>R}|\phi'(x)/\phi(x)|\leq1~~ \mbox{.}
\end{equation}

Observe that $\Phi_{\alpha,R}$ is still in $H^1(\mathbb{R}^n) \cap D(V)$, as both $\chi_R$ and $\nabla\chi_R$ are bounded, and that $\nabla\chi_R$ is compactly supported. Then, the inner product in (\ref{eq:inner_p}) with $\Phi_\alpha$ replaced by $\Phi_{\alpha,R}$ is well defined in the sense of quadratic forms and can be re-estimated:
\begin{equation}
    \begin{aligned}
    \Re\langle \Phi_{\alpha,R}, (H_{f_\alpha}-E)\Phi_{\alpha,R}\rangle &= \langle \Phi_{\alpha,R}, (V-|\omega_\alpha|^2-E)\Phi_{\alpha,R}\rangle\\
    &\geq \langle \Phi_{\alpha,R}, (V- \left\vert\frac{\phi'(f_\alpha)}{\phi(f_\alpha)}\right\vert^2(1-\epsilon)(V-E)_+-E)\Phi_{\alpha,R}\rangle\\
    &\geq \langle \phi(f_\alpha)\psi, (V-(1-\epsilon)(V-E)_+-E)\phi(f_\alpha)\psi\rangle~,
    \end{aligned}
\end{equation}
where we have used that $\chi_R^2|\phi'/\phi|^2\leq 1$. Hence, we no longer need to adjust $\epsilon$ to cause the multiplicative factor of $(V-E)_+$ in (\ref{eq:inner_p}) be less than 1, i.e. $\eta_\epsilon$ in (\ref{eq:inner}) is now replaced by $\epsilon$, which, in turn, replaces the conclusion of Lemma \ref{lm:1} with:
\begin{equation}\label{eq:lm2_improved}
    ||\Phi_{\alpha,R} ||^2_2 \leq \frac{1}{\epsilon\delta}\left\{\Re\langle \Phi_{\alpha,R} , (H_{f_\alpha}-E)\Phi_{\alpha,R} \rangle+\int|\Phi_{\alpha,R} |^2(V-E)_-\right\}+\int_{\{V\leq E+\delta\}}|\Phi_{\alpha,R}|^2~,
\end{equation}

However, this comes at the cost that (\ref{eq:re_0}) is no longer true, i.e. 
\begin{equation}
\Re\langle \Phi_{\alpha,R}, (H_{f_\alpha}-E)\Phi_{\alpha,R}\rangle\neq 0 ~\mbox{.}
\end{equation}
We account for this through the following Lemma, which generalizes \cite[Lemma 3.7]{Hislop_Sigal_book_1996}. 
\begin{lemma}\label{lm:2}
Given the setup outlined above, with $\Phi_{\alpha,R}$ defined in (\ref{eq:phi_ar}), we have
\begin{equation}\label{eq:l2}
    \Re\langle\Phi_{\alpha,R},(H_{f_\alpha}-E)\Phi_{\alpha,R}\rangle =\langle\xi_\alpha \phi(f_\alpha)^2\psi,\psi\rangle~,
\end{equation}
where 
\begin{equation}\label{eq:xi}
    \xi_\alpha := |\nabla \chi_R|^2 +2(\nabla \chi_R\cdot\nabla f_\alpha)\chi_R\frac{\phi'(f_\alpha)}{\phi(f_\alpha)}~.
\end{equation}
\end{lemma}
Here, as mentioned earlier, $\Phi_{\alpha,R} \in H^1(\mathbb{R}^n) \cap D(V)$, implies that the inner product in (\ref{eq:l2}) is well-defined in the sense of quadratic forms.
\begin{proof}
Since, $(H-E) \psi = 0$, we obtain for the left side of (\ref{eq:l2}):
\begin{equation}\label{eq:12}
    \begin{aligned}
    \langle\Phi_{\alpha,R},(H_{f_\alpha}-E)\Phi_{\alpha,R}\rangle& = \langle\chi_R \phi(f_\alpha)\psi,\phi(f_\alpha)(H-E)\chi_R\psi\rangle \\
    &= \langle \chi_R \phi(f_\alpha)^2\psi,(-\Delta\chi_R -2\nabla\chi_R\cdot\nabla)\psi\rangle
    \end{aligned}
\end{equation}
Using integration by parts, we compute
\begin{align}
    \langle \chi_R & \phi(f_\alpha)^2\psi, -2\nabla\chi_R\cdot\nabla\psi\rangle = \langle [ 2(\Delta\chi_R)\chi_R\phi(f_\alpha)^2 + 2|\nabla\chi_R|^2\phi(f_\alpha)^2 \nonumber \\
    &  + 4(\nabla\chi_R\cdot\nabla f_\alpha)\chi_R \phi(f_\alpha)\phi'(f_\alpha) ]\psi,\psi\rangle +\langle 2\chi_R \phi(f_\alpha)^2 \nabla\chi_R \cdot \nabla\psi,\psi\rangle ~\mbox{,} \label{eq:lemma_commutator}
\end{align}
which, since the last term in (\ref{eq:lemma_commutator}) is antisymmetric, yields
\begin{align*}
    &~~~\Re\langle -2(\nabla\chi_R)\chi_R \phi(f_\alpha)^2\psi,\nabla\psi\rangle\\
    &=\langle [ (\Delta\chi_R)\chi_R\phi(f_\alpha)^2 +|\nabla\chi_R|^2\phi(f_\alpha)^2 + 2(\nabla\chi_R\cdot\nabla f_\alpha)\chi_R \phi(f_\alpha)\phi'(f_\alpha) ]\psi,\psi\rangle
\end{align*}
Combining this with (\ref{eq:12}), we thus arrive at the claim:
\begin{equation*}
    \Re\langle\Phi_\alpha,(H_{f_\alpha}-E)\Phi_\alpha\rangle =\langle [ |\nabla\chi_R|^2\phi(f_\alpha)^2 + 2(\nabla\chi_R\cdot\nabla f_\alpha)\chi_R \phi(f_\alpha)\phi'(f_\alpha) ]\psi,\psi\rangle ~\mbox{.}
\end{equation*}
\end{proof}

Lemma \ref{lm:2} allows to prove Theorem \ref{the:2}, which, by requiring the additional hypotheses in (H\ref{hyp:3}) on the both the weight (\ref{eq:p'p}) and the behavior of the Agmon metric at infinity (\ref{eq:rho2inf}), removes the positive lower bound on $\epsilon$ imposed in (\ref{eq:epsilon}):

\begin{proof}[Proof of Theorem \ref{the:2}]
First, notice that since $\chi_R^2 \leq 1$, the upper bounds on the right most sides of (\ref{eq:bd1})-(\ref{eq:bd2}) remain unaffected when replacing $\Phi_{\alpha}$ with $\Phi_{\alpha,R}$.
Thus, combining (\ref{eq:lm2_improved}) with the bounds in (\ref{eq:bd1})-(\ref{eq:bd2}), we have
\begin{equation}
    \begin{aligned}
    ||\Phi_{\alpha,R}||^2_2  & \leq\frac{1}{\epsilon\delta}\Re\langle \Phi_{\alpha,R}, (H_{f_\alpha}-E)\Phi_{\alpha,R}\rangle +\frac{C_{1}}{\epsilon\delta}+C_{2}~.
    \end{aligned}
\end{equation}
We recall that the constants $C_1$ and $C_2$ in (\ref{eq:bd1})-(\ref{eq:bd2}) are independent of $\alpha > 0$.

Then, applying Lemma \ref{lm:2}, we obtain
\begin{align*}
        ||\Phi_{\alpha,R}||^2_2 & \leq \frac{1}{\epsilon\delta}\langle\xi_\alpha \phi(f_\alpha)^2\psi,\psi\rangle +\frac{C_{1}}{\epsilon\delta}+C_{2} \\
        &\leq \frac{||\psi||^2_2}{\epsilon\delta}\sup_{x\in \mathrm{supp}|\nabla\chi_R|}\left\{\xi_\alpha \phi(f_\alpha)^2\right\} +\frac{C_{1}}{\epsilon\delta}+C_{2} \\
        &=\frac{||\psi||^2_2}{\epsilon\delta}\sup_{x\in \mathrm{supp}|\nabla\chi_R|}\left\{\left[|\nabla \chi_R|^2 +2|\nabla \chi_R | |\nabla f_\alpha|\chi_R\frac{\phi'(f_\alpha)}{\phi(f_\alpha)}\right] \phi(f_\alpha)^2\right\} +   \frac{C_{1}}{\epsilon\delta}+C_{2} \\
        &\leq \frac{||\psi||^2_2}{\epsilon\delta}\sup_{x\in \mathrm{supp}|\nabla\chi_R|}\left\{[|\nabla \chi_R|^2+2 |\nabla\chi_R| |\nabla f_0| \phi(f)^2\right\} +\frac{C_{1}}{\epsilon\delta}+C_{2} =: a_{\epsilon,\delta} < \infty~,
\end{align*}
where, in the last step, we used (\ref{eq:defnR}) and that $|\nabla f_\alpha|^2 \leq |\nabla f_0|^2$, see (\ref{eq:df_a}). Again, since the constant $a_{\epsilon,\delta}$ is $\alpha$-independent, Fatou's Lemma thus yields
\begin{equation} \label{eq:proof_modf_thm_boundnoncompact}
    ||\Phi_{R}||^2_2:=\lim_{\alpha\to0^+}||\Phi_{\alpha,R}||^2_2\leq a_{\epsilon,\delta} <\infty~.
\end{equation}

In summary, using (\ref{eq:proof_modf_thm_boundnoncompact}) and (\ref{eq:bd2}), we can bound the integral in the claim:
\begin{align*}
    \int \phi & ((1-\epsilon)\rho_E(x))^2|\psi(x)|^2 ~\ud^n x  =\left\{ \int_{\{V > E+\delta\} }+\int_{\{V \leq E+\delta \}} \right\}\phi((1-\epsilon)\rho_E(x))^2|\psi(x)|^2 ~\ud^n x\\
    &\leq \int_{ \{V > E+\delta \}} \phi((1-\epsilon)\rho_E(x))^2|\psi(x)|^2 ~\ud^n x+ C_2 \\
    &= \int_{ \{V > E+\delta \}} (1-\chi^2_R)\phi((1-\epsilon)\rho_E(x))^2|\psi(x)|^2 ~\ud^n x+||\Phi_{R}||^2_2 + C_2\\
    &\leq \Vert \psi \Vert_2^2 \sup_{x \in \overline{B_{R+1}(0)}} \left\{\phi((1-\epsilon)\rho_E(x))^2 \right\} + a_{\epsilon, \delta} + C_2 ~\mbox{,}
\end{align*}
where the first term in the last step is finite because it is the supremum of a continuous function over a compact domain. 
\end{proof}

\subsection{Concluding remarks} \label{sec:L2:sub_conclremarks}

We conclude this section with two remarks about about the scope of Theorems \ref{the:1} and \ref{the:2}: The first relates $\delta$ in the integrability condition (\ref{eq:l_2}) to the distance of the eigenvalue under consideration to the essential spectrum, the second provides examples of potentials which satisfy the hypotheses of Theorems \ref{the:1}-\ref{the:2}.

\subsubsection{Integrability condition and the essential spectrum}  \label{sec:L2:sub_conclremarks_ess}
The classical results for exponential decay of the wave function, both isotropic \cite{Froese_Herbst_CMP_1982, BCH_HelvPhysActa_1997, Combes_Thomas_CMP_1973, Slaggie_Wichmann_JMP_1962} and anisotropic \cite{Agmon_lectures1982}, consider eigenvalues $E$ below the bottom of the essential spectrum $\sigma_{\mathrm{ess}}$ of the Schr\"odinger operator. Indeed, Agmon decay in its most basic form, assumes compactness of the enlarged classically allowed region $\{ V \leq E  + \delta \}$, for some $\delta > 0$, which in turn implies that $E \leq \inf \sigma_{\mathrm{ess}} -\delta$ (e.g. use the well-known Persson characterization for the essential spectrum \cite{Persson_1960_essentialsp} (see also \cite[Theorem 14.11, chapter 14.4]{Hislop_Sigal_book_1996}). The following shows that $\delta$ in the integrability condition (\ref{eq:l_2}) of Theorems \ref{the:1}-\ref{the:2} serves a similar purpose, specifically:
\begin{prop} \label{prop:essentialspectrum}
Given a Schr\"odinger operator $H = -\Delta+V$ where
\begin{enumerate}
    \item[(i)] $V$ is real-valued and bounded below with $V\geq\inf_{x\in\mathbb{R}^n}{V}=:m_V>-\infty$ and
    \item[(ii)] $V$ is relatively $(-\Delta)$-bounded with relative bound strictly less than one, i.e.
    \begin{equation}\label{eq:re_bd}
        \lim_{\lambda\to + \infty}||V (-\Delta+\lambda)^{-1}||<1~,
    \end{equation}
\end{enumerate}
in particular, $H$ is self-adjoint on the domain $H^2(\mathbb{R}^n)$. 

Suppose that for some $E_0 \in\mathbb{R}$, there exists $\delta>0$ such that the set
\begin{equation}\label{eq:A}
    A_{E_0;\delta}:=\{x\in\mathbb{R}^n:V \leq E_0+\delta\}
\end{equation}
has {\em{finite}} Lebesgue measure. Then, one has
\begin{equation}
    \sigma_{ess}(H) \subseteq [E_0 + \delta, + \infty) ~\mbox{.}
\end{equation}
\end{prop}
In view of Theorems \ref{the:1}-\ref{the:2}, observe that, since
\begin{equation}
\chi_{\{V\leq E+\delta\}} \leq \chi_{\{V\leq E+\delta\}}\phi\left((1-\epsilon)\rho_E\right) ~\mbox{,}
\end{equation}
Proposition \ref{prop:essentialspectrum} interprets $\delta$ in the integrability condition (\ref{eq:l_2}) as a lower bound on the distance of the eigenvalue $E$ to the bottom of the essential spectrum of $H$.

The idea of the proof of Proposition \ref{prop:essentialspectrum} is to show that modifying the potential appropriately on the finite measure set $A_{E_0;\delta}$ does not alter the essential spectrum. To keep the paper self-contained, we include the brief argument as follows:
\begin{proof}
Consider the perturbation
\begin{equation}
    W:=\chi_{\{V\leq E_0+\delta\}}\cdot(E_0+\delta - V )~,
\end{equation}
so that $0 \leq W \leq E_0 + \delta - m_V$, which by (\ref{eq:A}), yields $W\in L^2 \cap L^\infty(\mathbb{R}^n)$ with
\begin{equation}
    ||W||_{2} \leq (E_0+\delta - m_V) |A_{E_0,\delta}| ~.
\end{equation}

%Since $W\in L^2 \cap L^\infty(\mathbb{R}^n)$ and since $(1 + \vert k \vert^2)^{-1} \in L^\infty(\mathbb{R}^n)$ which vanishes at infinity, we obtain that $W$ is relatively compact with respect to $(-\Delta)$. 

Since $W\in L^2(\mathbb{R}^n)$, we obtain that $W$ is relatively compact with respect to $(-\Delta)$ (use e.g. Theorem 14.9 in \cite{Hislop_Sigal_book_1996}). Using (\ref{eq:re_bd}), this shows that $W$ is also relatively compact with respect to $H$. In particular, we conclude that
\begin{equation}\label{eq:sigma}
    \sigma_{\rm ess}{(H)} = \sigma_{\rm ess}{(H+W)}~.
\end{equation}

By construction, one has $V+W\geq E_0+\delta$, whence (\ref{eq:sigma}) yields
\begin{equation}
    \sigma_{\rm ess}{(H)} = \sigma_{\rm ess}{(H+W)}\subseteq \sigma{(H+W)}\subseteq[E_0+\delta,+\infty)~,
\end{equation}
proving the claim.
\end{proof}

\subsubsection{Example for potentials satisfying the integrability condition} \label{sec:L2:sub_conclremarks_example}
Our second remark constructs explicit examples in dimension one ($n=1$) for potentials which satisfy the hypotheses of Theorems \ref{the:1}-\ref{the:2}, specifically the integrability condition in (\ref{eq:l_2}). In particular, this shows that these results apply to a nonempty class of Schr\"odinger operators. In the following construction, the dimension ($n=1$) mainly enters through the simple form of the Agmon metric on $\mathbb{R}$, see (\ref{eq:example_1d_1}) - (\ref{eq:example_1d_2}) below.

In one dimension, take a potential $V_0\in \mathcal{C}(\mathbb{R})$ with $V_0 \leq 0$, $\lim_{x\to\infty}V_0(x) = 0$, and so that its negative discrete spectrum satisfies $\sigma_{\rm disc} (-\Delta+V_0) \cap (-\infty, 0) \neq \emptyset $; in particular, then $\sigma_{\rm ess}(-\Delta+V_0) = [0,+\infty)$. We thus have a negative eigenvalue $E_0$ of $(-\Delta+V_0)$, $-\infty< m_{V_0} :=\inf V_0 < E_0 < 0$, with associated eigenfunction $\psi_0$, $||\psi_0||_2 = 1$. 

Given a monotonic weight function $\phi \in \mathcal{C}^1([0,+\infty))$ for which the conditions (\ref{eq:phi_1}) and (\ref{eq:mphi}) hold, take a continuous $V_1 \in L^2 \cap L^\infty(\mathbb{R})$, $V_1 \leq 0$, such that $V:=V_0+V_1$ satisfies
\begin{equation} \label{eq:example_construction1}
m_{V_0} \leq V \leq 0 ~\mbox{, } ~\{x \in \mathbb{R} ~:~ V(x) = m_{V_0} \} ~\mbox{is unbounded} ~\mbox{, }
\end{equation}
and that
\begin{equation}\label{eq:finite}
    \int \chi_{\{V \leq  E_0 \}}(x) ~ \phi(|m_{V_0} |^{1/2}|x|)^2 ~\ud x <\infty ~.
\end{equation}
Note that since $m_{V_0} < E_0 < 0$ and $\lim_{x\to\infty}V_0(x) = 0$, the set $\{ V_0 \leq E_0 \}$ is bounded, whence such $V_1$ can always be found independent of how $V_0$ decays to zero. Morally, the perturbation $V_1$ is used to add negative spikes to the original potential $V_0$ reaching down to $m_{V_0}$ on the set $\{x \in \mathbb{R} ~:~ V_0 > E_0\}$, so that $V_0 + V_1$ satisfies the condition in (\ref{eq:finite}).
\begin{remark}
To construct $V_1$ more explicitly, fix $0 < R < +\infty$ such that
\begin{equation}
\{x \in \mathbb{R} ~:~ V_0(x) \leq E_0 \} \subseteq [-\frac{R}{2} , \frac{R}{2}] ~\mbox{.}
\end{equation}

Take $(c_j)_{j \in \mathbb{N}}$ to be an increasing sequence in the set $\{x \in (0, + \infty) ~:~ V_0(x) > E_0 \}$ such that $\lim_{j \to \infty} c_j = + \infty$. For $j \in \mathbb{N}$, let $0 < l_j < 1$ with $\sum_{j \in \mathbb{N}} l_j < \infty$, so that the collection of intervals $[c_j - \frac{l_j}{2}, c_j + \frac{l_j}{2}]$ are mutually disjoint and that
\begin{equation}
[c_j - \frac{l_j}{2}, c_j + \frac{l_j}{2}] \subseteq \{x \in (0, + \infty) ~:~ V_0(x) > E_0 \} ~\mbox{.}
\end{equation}

Now, obtain $V \in \mathcal{C}(\mathbb{R})$ by modifying $V_0$ continuously on each of the mutually disjoint intervals $[c_j - \frac{l_j}{2}, c_j + \frac{l_j}{2}]$ so that, for $j \in \mathbb{N}$, one has
\begin{align} \label{eq:constructionV_mod-1}
V(x) = m_{V_0} ~\mbox{, for } x \in [c_j - \frac{l_j}{4}, c_j + \frac{l_j}{4}] ~\mbox{,} \\
m_{V_0} \leq V(x) \leq 0 ~\mbox{, for } x \in [c_j - \frac{l_j}{2}, c_j + \frac{l_j}{2}] ~\mbox{,} 
\end{align}
while not changing $V_0$ otherwise, i.e.
\begin{equation} \label{eq:constructionV_mod}
V(x) = V_0(x) ~\mbox{, for } x \in \mathbb{R} \setminus \left\{ \bigcup_{j \in \mathbb{N}} [c_j - \frac{l_j}{2}, c_j + \frac{l_j}{2}] \right\} ~\mbox{.}
\end{equation}
Set $V_1:= V - V_0 \in L^{\infty}(\mathbb{R})$. Since $m_{V_0} \leq V_1 \leq 0$, the properties of $V$ in (\ref{eq:constructionV_mod-1})-(\ref{eq:constructionV_mod}) imply
\begin{equation} \label{eq:example_summcond-1}
\Vert V_1 \Vert_2^2 \leq m_{V_0}^2 \sum_{j \in \mathbb{N}} l_j < +\infty ~\mbox{,}
\end{equation} 
whence $V_1 \in L^2 \cap L^\infty(\mathbb{R})$ as claimed. By construction, we also obtain
\begin{equation}
\bigcup_{j \in \mathbb{N}} [c_j - \frac{l_j}{4}, c_j + \frac{l_j}{4}] \subseteq \{x \in \mathbb{R} ~:~ V(x) = m_{V_0} \} ~\mbox{,}
\end{equation}
verifying (\ref{eq:example_construction1}). Finally, since
\begin{equation}
\{V \leq  E_0 \} \subseteq \bigcup_{j \in \mathbb{N}} [c_j - \frac{l_j}{2}, c_j + \frac{l_j}{2}] \cup [-\frac{R}{2}, \frac{R}{2}] ~\mbox{,}
\end{equation}
we conclude, in analogy to (\ref{eq:interpretintegrabilitycond_1d_summability}) in Remark \ref{rem:integrabilitycond_1d}, that
\begin{align} \label{eq:example_summcond}
\int \chi_{\{V \leq  E_0 \}}(x) ~&  \phi(|m_{V_0} |^{1/2}|x|)^2 ~\ud x \leq  R \cdot \phi\left(|m_{V_0} |^{1/2} R \right)^2 \nonumber \\
   & + \sum_{j \in \mathbb{N}} l_j \cdot \phi\left(|m_{V_0} |^{1/2} \left(c_j + \frac{1}{2} \right) \right)^2 ~\mbox{,}
\end{align}
where we used that $\phi$ is increasing and that $0< l_j < 1$, for all $j \in \mathbb{N}$.

In summary, choosing the sequence $(l_j)$ to decay fast enough (adapted to the rate of increase of $\phi$) such that the right hand side of (\ref{eq:example_summcond}) is finite, we obtain $V$ (and thus implicitly $V_1= V - V_0$) such that (\ref{eq:finite}) holds. Here, we note that since by hypothesis $\phi$ has bounded logarithmic derivative (\ref{eq:mphi}), i.e. by Gr\"onwall's inequality $\phi(t) \leq \phi(0) \mathrm{e}^{M_\phi t}$ for $t \in [0,+\infty)$, (\ref{eq:example_summcond}) is always satisfied if
\begin{equation}
l_j \leq \frac{1}{j^2} \mathrm{e}^{-2 M_\phi \vert m_{V_0} \vert^{1/2} (c_j + \frac{1}{2})} ~\mbox{;}
\end{equation}
of course, a slower decay of $(l_j)$ is possible depending on the rate of increase of $\phi$.
\end{remark}

Since $V_1\in L^2(\mathbb{R})$, it is relatively compact to $(-\Delta+V_0)$. Hence, we know that 
\begin{equation}
    \sigma_{\rm ess}(-\Delta+V_0+V_1) = \sigma_{\rm ess}(-\Delta+V_0) = [0,\infty)~.
\end{equation}

Moreover, by construction we have
\begin{align}
m_{V_0} < \langle \psi_0 , H \psi_0 \rangle & = \langle \psi_0 , (-\Delta + V_0) \psi_0 \rangle + \underbrace{\langle \psi_0, V_1 \psi_0 \rangle}_{< 0} < E_0 ~\mbox{,}
\end{align}
thus the min-max principle guarantees the existence of an eigenvalue $E$ of $H$ with $m_{V_0} < E \leq E_0$ of $H$ and an associated eigenfunction $\psi$, $||\psi||_2=1$. Since $V$ is bounded, the operator $H=-\Delta + V$ is self-adjoint on the domain $D(H)=H^2(\mathbb{R}) \subseteq D(V)$, where $D(V)$ is defined in (H\ref{hyp:1}). In particular, by the Sobolev Embedding Theorem, the hypothesis $\psi \in L^\infty(\mathbb{R})$ is trivially satisfied.

Now, we can bound the Agmon distance to the origin by the following:
\begin{equation} \label{eq:example_1d_1}
    \rho_E(x) = \mathrm{sgn}(x) \int_0^x(V(t)-E)_+^{1/2} ~\ud t \leq |m_{V_0}|^{1/2}|x|~, x \in \mathbb{R} ~\mbox{,}
\end{equation}
which, together with (\ref{eq:finite}), implies that for every $0 < \epsilon < 1$,
\begin{align} \label{eq:example_1d_2}
    ||\chi_{\{V_0+V_1 \leq E \}}\phi((1 -\epsilon) \rho_E)||_2 & \leq ||\chi_{\{V_0+V_1 \leq E_0 \}}\phi((1 -\epsilon) \rho_E)||_2 \nonumber \\
       & \leq||\chi_{\{V_0+V_1 \leq E_0 \}}\phi(|m_V|^{1/2}|x|)||_2<\infty~.
\end{align}
Hence, the eigenfunction $\psi$ associated with $E$ satisfies the hypotheses of Theorem \ref{the:1} for every $0< \delta < E_0 - E$. Notice that by (\ref{eq:example_construction1}), since $m_{V_0} < E$, the set 
\begin{equation}
\{x \in \mathbb{R} ~:~ V(x) = m_{V_0} \} \subseteq \{x \in \mathbb{R} ~:~ V(x) \leq E \}
\end{equation} 
is unbounded, in particular, the classically allowed region for the potential $V$ and the eigenvalue $E$ of $H$ is not compact.

Finally, since by Remark \ref{remark:agmon1d_infty}, (\ref{eq:rho2inf}) in (H\ref{hyp:3}) always holds for dimension $n=1$, the non-emptiness of the content of Theorem \ref{the:2} follows for all admissible weights $\phi$ which also satisfy the condition in (\ref{eq:p'p}).

\section{Point-wise decay} \label{sec:ptw}

In this last section we show that the $L^2$ decay in Theorem \ref{the:1} and \ref{the:2} implies point-wise decay {\em{if the potential is sufficiently regular}}. The case where the weight function $\phi$ is exponential is well known \cite[chapter 5]{Agmon_lectures1982}, see also \cite[chapter 3.5]{Hislop_Sigal_book_1996}. Here, we will modify these arguments appropriately to obtain point-wise bounds for general admissible weights $1 \leq \phi \in \mathcal{C}^1([0,+\infty))$ with bounded logarithmic derivative. The latter in particular accounts for the power law weights $\phi(t) = (1 + t)^r$ with $r >0$. In the following we will consider potentials $V \in C_b^k(\mathbb{R}^n)$, for $k \in \mathbb{N}$, where, as usual, $C_b^k(\mathbb{R}^n)$ denotes the $\mathcal{C}^k$-functions whose partial derivatives, up to order $k$, are all bounded.

Since, for all $s \geq 0$, $(-\Delta + 1)^{-1}: H^{s}(\mathbb{R}^n) \to H^{s+2}(\mathbb{R}^n)$, regularity of the potential $V \in C_b^k(\mathbb{R}^n)$ implies that any eigenfunction $\psi \in L^2(\mathbb{R})$ of $H = -\Delta + V$ automatically satisfies $\psi \in H^{k +2}(\mathbb{R}^n)$. In particular, taking $k \in \mathbb{N}$ sufficiently large to ensure $k + 2 > \frac{n}{2}$ or equivalently $k > (n-4)/2$, one can then prove local bounds for the $L^\infty$-norm of $\psi$ by its $L^2$-norm, specifically:
%\begin{lemma}[cf. \cite[Theorem 5.1]{Agmon_lectures1982} , \cite[Lemma 3.9]{Hislop_Sigal_book_1996} ]\label{lm:3}
%\begin{lemma}[cf. \cite[Lemma 3.9]{Hislop_Sigal_book_1996}] \label{lm:3}
\begin{lemma}[ cf. Theorem 5.1 in \cite{Agmon_lectures1982} or Lemma 3.9 in \cite{Hislop_Sigal_book_1996}   ] \label{lm:3}
Let $V \in C^k_b(\mathbb{R}^n)$ for $k \in \mathbb{N}$ with $k > (n-4)/2$. Suppose $H\psi = E\psi$ for $\psi \in L^2(\mathbb{R}^n)$ and $E \in \mathbb{R}$. Then, $\psi \in \mathcal{C}_b^0(\mathbb{R}^n)$ vanishing at infinity and there exists $C_{E,V}$ depending on E and $\sup_{x\in\mathbb{R}^n}|V^{(\alpha)}(x)|$, $|\alpha| = 0,...,k$, such that for all $x_0\in\mathbb{R}^n$, one has
\begin{equation}\label{eq:lm3}
    \max_{x\in B_{1/2}(x_0)}{|\psi(x)|}\leq C_{E,V}||\psi||_{L^2(B_1(x_0))}~.
\end{equation}
\end{lemma}
Here, we mention that the condition $k + 2 > \frac{n}{2}$ ensures that $\widehat{\psi} \in L^1(\mathbb{R}^n)$, in particular $\psi \in \mathcal{C}_b^0(\mathbb{R}^n)$ vanishing at infinity, so that the boundedness condition on $\psi$ in (H\ref{hyp:1}) part (a) holds true for any such eigenfunction, irrespective of the dimension $n$; see also Remark \ref{rem:linftyassumption}. Relying on Lemma \ref{lm:3}, we can then claim:
\begin{theorem}[Point-wise Bound]\label{thm:4}
Let $H = -\Delta + V$ be as in (H\ref{hyp:1}) and assume that moreover $V \in C^k_b(\mathbb{R}^n)$ for $k \in \mathbb{N}$ with $k > (n-4)/2$. Let $\psi \in L^2(\mathbb{R}^n)$ be an eigenfunction of $H$ with associated eigenvalue $E \in \mathbb{R}$ such that, for some weight function $1 \leq \phi \in \mathcal{C}^1([0,+\infty))$ satisfying (\ref{eq:mphi}) and $0 < \epsilon < 1$, one has
\begin{equation*}
    \phi((1-\epsilon)\rho_E) \psi \in L^2(\mathbb{R}^n) ~\mbox{.}
\end{equation*}

Then, there exists a constant $0< C_\epsilon < \infty$ such that 
\begin{equation} \label{eq:thm4}
    |\psi(x)| \leq C_{\epsilon} ~\phi((1-\epsilon)\rho_E(x))^{-1}~, ~\forall x\in\mathbb{R}^n ~\mbox{.}
\end{equation}
\end{theorem}

\begin{proof}[Proof of Theorem \ref{thm:4}]
Fix an arbitrary $x_0\in\mathbb{R}^n$. Then, by Lemma \ref{lm:3}, there exists $C_{E,M,V}$ such that (\ref{eq:lm3}) holds, whence since $\phi \geq 1$, we have uniformly over all $x\in B_{1/2}(x_0)$:
\begin{align}
       ||\psi & \phi((1-\epsilon) \rho_E)||_{\infty; B_{1/2}(x_0)} \leq C_{E,V}\left( \sup_{x\in B_1(x_0)}\phi((1-\epsilon)\rho_E ) \right)||\psi||_{L^2(B_1(x_0))} \nonumber \\
        &\leq C_{E,V}\left( \sup_{x, y\in B_1(x_0)}\frac{\phi((1-\epsilon)\rho_E(x))}{\phi((1-\epsilon)\rho_E(y))} \right) \cdot||\psi\phi((1-\epsilon)\rho_E)||_{L^2(B_1(x_0))} \label{eq:pointwisebound_start}
\end{align}
By hypothesis in (\ref{eq:mphi}), $\phi$ has bounded logarithmic derivative, thus
\begin{equation}
 |\log(\phi(s))-\log(\phi(t))|\leq M_\phi \cdot |s-t|~, ~\forall s,t\in [0,\infty)~.
\end{equation}
In particular, for all $x,y \in B_1(x_0)$, we conclude
\begin{align}
    \left\vert\log\left(\frac{\phi((1-\epsilon)\rho_E(x))}{\phi((1-\epsilon)\rho_E(y))}\right)\right\vert &\leq M_\phi \cdot(1-\epsilon)|\rho_E(x)-\rho_E(y)|\\
    &\leq M_\phi \cdot(1-\epsilon)\rho_E(x,y)\\
    &\leq 2 M_\phi (1-\epsilon)c~, \label{eq:pointwise_metricbound}
\end{align}
where $c = (\max{V}-E)^{1/2}_+$ using that $V$ was assumed to be bounded. Observe that the upper bound in (\ref{eq:pointwise_metricbound}) is {\em{uniform}} in $x_0 \in \mathbb{R}^n$. 

In summary, we obtain
\begin{equation} \label{eq:pointwise_metricbound_1}
    \sup_{x,y\in B_1(x_0)}\left\vert\frac{\phi((1-\epsilon)\rho_E(x))}{\phi((1-\epsilon)\rho_E(y))}\right\vert \leq \mathrm{e}^{2M_\phi (1-\epsilon)c}~,
\end{equation}
which, combined with (\ref{eq:pointwisebound_start}), yields
\begin{align} \label{eq:pointwisebound_end}
    ||\psi & \phi((1-\epsilon)\rho_E)||_{\infty; B_{1/2}(x_0)} \leq C_{E,V}\mathrm{e}^{2M_\phi(1-\epsilon)c}\cdot||\psi\phi((1-\epsilon)\rho_E)||_{L^2(B_1(x_0))} \nonumber \\
    &\leq C_{E,V}\mathrm{e}^{2M_\phi(1-\epsilon)c}\cdot||\psi\phi((1-\epsilon)\rho_E)||_{L^2(\mathbb{R}^n)} ~\mbox{.}
\end{align}
Since $x_0 \in \mathbb{R}^n$ was arbitrary and the right-most side of (\ref{eq:pointwisebound_end}) is independent of $x_0$, we obtain the claim in (\ref{eq:thm4}). 
\end{proof}

%%%%%%%%%%%%%%%%%%%%%%%%%%%%%%%%%%%%%%%%%%%%%%%%%%%%%%%%%%%%%%%%%%%%%%%%%%%%%%%%%%%%%%%%%%%%

\end{document}